\documentclass{article}

\RequirePackage[OT1]{fontenc}
\RequirePackage[
amsthm,amsmath,natbib]{imsart}
\usepackage{xcolor}
\usepackage[normalem]{ulem}

\usepackage{graphicx}

\usepackage{latexsym,amsmath}
\usepackage{amsmath,amsthm,amscd}
\usepackage{amsfonts}
\usepackage{calrsfs}
\usepackage[psamsfonts]{amssymb}
\usepackage{pstricks,pst-plot}
\usepackage{enumerate}

\usepackage{url}

\usepackage{tocvsec2}

\usepackage{epstopdf}

\usepackage{wrapfig}
\usepackage{float}

\usepackage[active]{srcltx}

\theoremstyle{plain} 
\newtheorem{theorem}{Theorem}

\theoremstyle{definition} 

\theoremstyle{definition} 

\newtheorem*{ex*}{Example}
\theoremstyle{remark} 

\theoremstyle{remark} 
\newtheorem{remark}[theorem]{Remark}
\newtheorem*{remark*}{Remark}

\newcommand{\mathsym}[1]{{}}
\newcommand{\unicode}[1]{{}}

\newcommand{\iincludegraphics}[1]{{}}

\renewcommand{\le}{\leqslant}
\renewcommand{\ge}{\geqslant}

\newcommand{\ii}[1]{\operatorname{I}\left\{#1\right\}}

\newcommand{\ch}{\operatorname{ch}}
\newcommand{\sh}{\operatorname{sh}}

\newcommand{\R}{\mathbb{R}}

\newcommand{\E}{\operatorname{\mathsf{E}}}
\renewcommand{\P}{\operatorname{\mathsf{P}}}

\newcommand{\de}{\delta}

\newcommand{\si}{\sigma}

\newcommand{\td}{\tilde d}

\newcommand{\p}[1]{\partial_{#1}}

\begin{document}

\begin{frontmatter}

\title{An exact bound on the truncated-tilted mean for symmetric distributions
}
\runtitle{Truncated-tilted mean, symmetric case
}
\date{\today}

\begin{aug}
\author{\fnms{Iosif} \snm{Pinelis}\ead[label=e1]{ipinelis@mtu.edu}}
\runauthor{Iosif Pinelis}

\affiliation{Michigan Technological University}

\address{Department of Mathematical Sciences\\
Michigan Technological University\\
Houghton, Michigan 49931, USA\\
E-mail: \printead[ipinelis@mtu.edu]{e1}
}
\end{aug}

\begin{abstract}
An exact upper bound on the Winsorised-tilted mean, \break 
$
\frac{\E Xe^{h(X\wedge w)}}{\E e^{h(X\wedge w)}}$, of a symmetric random variable $X$ in terms of its second moment is given. 
Such results are used in work on nonuniform Berry--Esseen-type bounds for general nonlinear statistics. 
\end{abstract}


\setattribute{keyword}{AMS}{AMS 2010 subject classifications:}
 

\end{frontmatter}

\settocdepth{chapter}


\settocdepth{subsubsection}


Cram\'er's tilt transform of a random variable (r.v.) $X$ is a r.v.\ $X_h$ such that 
\begin{equation*}
	\E f(X_h)=\frac{\E f(X)e^{h\,X}}{\E e^{h\,X}}
\end{equation*}
for all nonnegative Borel functions $f$, where $h$ is a real parameter. This transform is an important tool in the theory of large deviation probabilities $\P(X>x)$, where $x>0$ is a large number; then the appropriate value of the parameter $h$ is positive.  
Unfortunately, if the right tail of the distribution of $X$ decreases slower than exponentially, then $\E e^{h\,X}=\infty$ for all $h>0$ and thus the tilt transform is not applicable. 
The usual recourse then is to replace $X$ in the exponent by its truncated counterpart, say $X\ii{X\le w}$ or $X\wedge w$, where $w$ is a real number. 
As shown in \cite{winzor,nonlinear}, of the two mentioned kinds of truncation, it is the so-called Winsorization, $X\wedge w$, of the r.v.\ $X$ that is more useful in the applications considered there. 

In particular, in \cite{nonlinear} one needs a good upper bound on the mean 
\begin{equation}\label{eq:E_h}
	\E_{h,w}X:=\frac{\E Xe^{h(X\wedge w)}}{\E e^{h(X\wedge w)}}. 
\end{equation}
of the Winsorised-tilted distribution of $X$. 
Note that $\E_{h,w}X$ is well defined and finite for any $h\in(0,\infty)$, any $w\in\R$, and any r.v.\ $X$ such that $\E(0\vee X)<\infty$. 

In \cite{winzor}, exact upper bounds on the denominator $\E e^{h(X\wedge w)}$ of the ratio in \eqref{eq:E_h} were provided, along with  
applications to pricing of certain financial derivatives. 

Take any positive real numbers $h$ and $w$. 
In \cite{pin11tilt}, exact upper bounds on $\E_{h,w}X$ given the first two moments of $X$. In particular, by \cite[Theorem 2.4(II)]{pin11tilt}, 
\begin{equation}\label{eq:asymm}
	\E_{h,w}X<\frac{e^{hw}-1}w\,\E X^2  
\end{equation}
for any real-valued r.v.\ with $\E X=0$ and $\E X^2\in(0,\infty)$; it is also shown in \cite{pin11tilt} that the factor $\frac{e^{hw}-1}w$ in \eqref{eq:asymm} is the best possible one.  

The purpose of this note is to show that in the case when (the distribution of) $X$ is symmetric, the factor $\frac{e^{hw}-1}w$ in \eqref{eq:asymm} can be improved to $\frac{\sh hw}w$; we write $\sh$ and $\ch$ in place of $\sinh$ and $\cosh$. 

\begin{theorem}\label{th:}
Let $X$ be any \emph{symmetric} real-valued r.v.\ with $\E X^2\in(0,\infty)$. Then 
\begin{equation}\label{eq:}
	0<\E_{h,w}X<\frac{\sh hw}w\,\E X^2. 
\end{equation}
\end{theorem}

\begin{remark}
The factor $\frac{\sh hw}w$ in \eqref{eq:} is the best possible one. More specifically, 
\begin{equation*}
	\lim_{\si\downarrow0}\,\frac1{\si^2}\,\sup\big\{\E_{h,w}X\colon\E X^2=\si^2,\ X\text{ is symmetric}\big\}=\frac{\sh hw}w. 
\end{equation*} 
In view of Theorem~\ref{th:}, this follows if we let $X$ take values $-w$, $0$, and $w$ with probabilities $\frac{\si^2}{2w^2}$, $1-\frac{\si^2}{w^2}$, and $\frac{\si^2}{2w^2}$, respectively, for $\si\in(0,w)$, and then let $\si\downarrow0$. 
Note here that the case of interest in applications in \cite{nonlinear} is precisely when $\E X^2$ is arbitrarily small. 
Also, in those applications $hw$ may be rather large, and then the symmetric-case factor $\frac{\sh hw}w$ will be almost twice as small as the general zero-mean-case factor $\frac{e^{hw}-1}w$. 
\end{remark}

\begin{proof}[Proof of Theorem~\ref{th:}] 
By \cite[Proposition~2.6(II)]{pin11tilt}, $\E_{h,w}X$ is increasing in $h>0$, so that $\E_{h,w}X>\E X=0$,  
and the first inequality in \eqref{eq:} follows. 

Let us prove the second inequality in \eqref{eq:}. 
By rescaling, without loss of generality (w.l.o.g.) $h=1$. 
For all real $x$ and $j\in\{0,1\}$, let 
\begin{equation*}
	f_j(x):=x^je^{x\wedge w}\quad\text{and}\quad 
	g_j(x):=\tfrac12\big(f_j(x)+f_j(-x)\big), 
\end{equation*}
using the convention $0^0:=1$;  
then 
\begin{equation}\label{eq:expec}
	\E g_j(|X|)=\E f_j(X)=\E X^je^{h(X\wedge w)}. 
\end{equation}
So, \eqref{eq:} will follow if one can show that 
\begin{equation}\label{eq:ineq}
	d:=d(u,v,w):=2\big[g_1(u)+g_1(v)-\tfrac{\sh hw}w\,\big(g_0(u)v^2+g_0(v)u^2\big)\big]<0 
\end{equation}
for all positive real $u,v,w$; 
indeed, then it will be enough to replace $u$ and $v$ in \eqref{eq:ineq} by independent copies (say $U$ and $V$) of the r.v.\ $|X|$, take the expectation, and use \eqref{eq:expec}. 
At this point, one should note that $d$ may equal $0$ if $u$ or $v$ equals $0$; in particular, $d=0$ if $u=0$ and $v=w$; however, the condition $\E X^2\in(0,\infty)$ in Theorem~\ref{th:} implies that $|X|>0$ with a nonzero probability, which will result in the second inequality in \eqref{eq:} being strict indeed. 

So, it remains to prove the inequality \eqref{eq:ineq}. Since $u$ and $v$ are interchangeable there, w.l.o.g.\ $0<u\le v$. Then (at least) one of the following three cases must occur: 
\begin{enumerate}[{}]
	\item\emph{Case 1}: $0<w\le u\le v$; 
	\item\emph{Case 2}: $0<u\le w\le v$; 
	\item\emph{Case 3}: $0<u\le v\le w$.  
\end{enumerate}
In each of these three cases, $d$ can be expressed without using the minimum function $\wedge$. 

In the subsequent treatment of each of these three cases, the default ranges of the variables $u$, $v$, and $w$ will be determined by the conditions of the case under consideration. For instance, if in Case~1 (say) an expression in $u,v,w$ is stated to be concave in $u$ or increasing in $v$, this will mean that it is concave in $u\in[w,v]$ (for any given $v$ and $w$ such that $0<w\le v$) or, respectively, increasing in $v\in[u,\infty)$ (for any given $u$ and $w$ such that $0<w\le u$). 

As usual, let $\p z$ denote the operator of partial differentiation with respect to a variable $z$. 

\subsubsection*{Case~1}
In this case, 
\begin{equation*}
d=	(e^w-e^{-u})u+(e^w-e^{-v})v
-\tfrac{\sh w}w\, \big(e^w (u^2+v^2)+e^{-v}u^2+e^{-u} v^2\big),   
\end{equation*}
whence  
\begin{equation}\label{eq:d2,d3}
\p v^2d=e^{-v} (2 - v) - \tfrac{\sh w}w (e^{-v} u^2 + 2 e^{-u} + 2 e^w)
\end{equation} 
and 
$\p v^3d=e^{-v} (v-3+\tfrac{u^2 \sh w}{w})$.  
So, $\p v^3d$ may change in sign at most once, and only from $-$ to $+$, if $v$ increases from $u$ to $\infty$. 
Therefore, 
\begin{equation}\label{eq:vee}
\p v^2d\le(\p v^2d)|_{v=u}\,\vee\,(\p v^2d)|_{v=\infty-}. 	
\end{equation}
Let $d_2:=d_2(u,w):=we^u\,(\p v^2d)|_{v=u}$. Then $d_2(u,0+)=0$ and 
$\p w d_2=\break
-2(e^{u + 2 w}-1) - u - (2 + u^2) \ch w<0$, so that $d_2<0$ or, equivalently, 
$(\p v^2d)|_{v=u}<0$. It is also clear from \eqref{eq:d2,d3} that $(\p v^2d)|_{v=\infty-}<0$. 
So, by \eqref{eq:vee}, $\p v^2d<0$ and hence $d$ is strictly concave in $v\in[u,\infty)$. 

Therefore, in Case~1 it suffices to show that $d|_{v=u}<0$ and $(\p v d)|_{v=u}<0$. 
Introduce $\td:=e^{u+w}\,\frac wu\,d$. 
Then $\p u^2(\td|_{v=u})=2 e^{u + 2 w} \big(w - (2 + u) \sh w\big)<0$, since $\sh w>w$. 
So, $\td|_{v=u}$ is strictly concave in $u$. 
Further, $(\td|_{v=u})|_{u=w}=\break
-(e^w-1)^3 (1 + e^w) w<0$. 

One can see that 
\begin{equation}\label{eq:...}
	\big(\p u(\td|_{v=u})\big)|_{u=w}
=1 + e^{2 w} \big(w + 2 e^w w - e^{2 w} (1 + w)\big)<0 
\end{equation}
for all $w>0$. 
Such inequalities, of the form $P(w,e^w)<0$ for some polynomial $P$ of two variables, can be proved in a rather algorithmic manner. Indeed, let $n\ge1$ be the degree of $P$ in $w$. ``Solving'' the inequality $P(w,e^w)<0$ for $w^n$, one can rewrite it as $\de(w):=w^n-P_1(w,e^w)<0$ or $\de(w)>0$ (depending on the sign of the coefficient of $w^n$ in $P$), where $P_1$ is some polynomial of degree $\le n-1$ in $w$. Then $\de'(w)$ will be a polynomial of degree $\le n-1$ in $w$, so that one can proceed by induction, ultimately reducing the problem to one on the sign of a polynomial in $e^w$ only. 
One can use a computer algebra system to (such as Mathematica) to execute such routine calculations, which appears to be a much more reliable and faster way to deal with such matters. 
In Mathematica, algorithms for solving inequalities like \eqref{eq:...} are implemented in the command \verb9Reduce9, which we indeed use to verify \eqref{eq:...}, as well as a few other similar inequalities. 
Similar methods were used e.g.\ in \cite{radem-asymp}. 

It follows that \rule{0pt}{10pt}$\td|_{v=u}<0$ and hence indeed $d|_{v=u}<0$. Now (in Case~1) it only remains to verify that $d_1:=d_1(u,w):=w e^u(\p v d)|_{v=u}<0$. 

Using again the inequality $w<\sh w$ (together with the conditions $0<w\le u$ of Case~1), one observes that 
\begin{align*}
\p u^2 d_1&=2 \sh w + e^{u + w} (w - 2 (2 + u) \sh w) \\ 
&\le2 \sh w + e^{2w} (w - 2 (2 + w) \sh w), 	
\end{align*}
and the latter expression can seen to be negative for all $w>0$ -- using again the command \verb9Reduce9, say. So, $d_1$ is concave in $u$. 
Yet another \verb9Reduce9 shows that $d_1|_{u=w}<0$ for $w>0$. 
Moreover, 
\begin{equation*}
(\p u d_1)|_{u=w}e^{-w}/2=(w-\sh w)\ch w - (\ch w + 2w \sh w)\sh w<0; 	
\end{equation*}
here we again used the inequality $w<\sh w$. 
This implies that indeed $d_1<0$, which completes the proof of \eqref{eq:ineq} in Case~1.  

\subsubsection*{Case~2}
In this case, 
\begin{equation*}
d=	2 u \sh u+v (\sh v+\sh w+\ch w)-v \ch v
-\tfrac{\sh w}w\, \big(u^2 (e^{-v}+e^w)+2 v^2 \ch u\big),   
\end{equation*}
whence, introducing  
\begin{equation}\label{eq:d1}
	d_1:=we^v\,\p v d, 
\end{equation}
one has   
\begin{align*}
e^{-v}\,\p v^2\,d_1
&=w \ch w + (w - 4 (2 + v) \ch u) \sh w \\
&\le w \ch w + (w - 4 (2 + w)) \sh w<0; 	
\end{align*}
the last inequality here can be obtained via another \verb9Reduce9, and 
the penultimate inequality follows by the condition $w\le v$ of Case~2.  
So, 
\begin{equation}\label{eq:d1 conc}
	\text{$d_1$ is concave in $v$.} 
\end{equation}
Note that the definition \eqref{eq:d1} of $d_1$, used in the present Case~2, differs from the definition of $d_1$ used in Case~1. 

Next, $(\p v d_1)|_{v=w}=e^w \big(e^w w-4 (w+1) \ch u\, \sh w\big)+w$ is obviously decreasing in $u>0$. 
So, $(\p v d_1)|_{v=w}<(\p v d_1)|_{v=w,u=0+}=3 w-e^{2 w} (w+2)+2$, which yet another \verb9Reduce9 shows to be negative for all $w>0$. 
Thus, 
\begin{equation}\label{eq:Dd1vw}
	(\p v d_1)|_{v=w}<0. 
\end{equation}

Now let us show that $d_1|_{v=w}<0$. One has 
\begin{equation}\label{eq:d1=...}
	d_1|_{v=w}=d_{11}+d_{12}\quad\text{and}\quad
	d_{11}=d_{111}+d_{112}, 
\end{equation}
where 
\begin{align}
d_{11}&:=e^w w (\sh w+\ch w-3 \ch u \sh w)+(w-1) w, \notag \\
d_{12}&:=(u^2-e^w w \ch u)\sh w , \label{eq:d12}\\
d_{111}&:=	e^w w (\ch w-2 \sh w)+(w-1) w, \notag \\
d_{112}&:=	-3 (\ch u-1) e^w w \sh w. \notag  
\end{align}
It is obvious that $d_{112}<0$. Also, $d_{111}<0$ by another \verb9Reduce9. 
Next, \break 
$\frac1{2\sh w}\p u(d_1|_{v=w})=u - 2 e^w w \sh u$. 
If $u\ge1/2$ then (by the condition $u\le w$ of Case~2) $w\ge1/2$, whence $2 e^w w \sh u>\sh u>u$, so that $\p u(d_1|_{v=w})<0$. 
Therefore, the condition $\p u(d_1|_{v=w})=0$ would imply $u<1/2$ and also $e^w w \sh u=u/2$, and then 
$u^2 -e^w w \ch u<u^2 -e^w w \sh u
=u^2 -u/2<0$, so that (by \eqref{eq:d12}) $d_{12}<0$ and hence, by \eqref{eq:d1=...},  $d_1|_{v=w}<0$. 
That is, $d_1|_{v=w}<0$ whenever $\p u(d_1|_{v=w})=0$. 

So, to prove the inequality $d_1|_{v=w}<0$ it is enough to verify that 
\begin{align*}
d_1|_{v=w,u=0+}&=(w-1) w+e^w w (\ch w-3 \sh w)<0 \quad\text{and} \\ \tfrac1w\,d_1|_{v=w,u=w}&=w+(w+e^w) \sh w-e^w (4 \sh w-1) \ch w-1<0,  
\end{align*}
which again can be done using \verb9Reduce9. 
We conclude that indeed $d_1|_{v=w}<0$. 

Using also the earlier established conditions \eqref{eq:d1 conc} and \eqref{eq:Dd1vw}, 
as well as the Case~2 condition $v\ge w$, one has $d_1<0$. So, by \eqref{eq:d1}, $d$ is decreasing in $v$. 

To complete the consideration of Case~2, it remains to show that $d|_{v=w}<0$. Observe here that 
\begin{align*}
\tfrac12\,\p u(d|_{v=w})&=\sh u+u \ch u-2\,\tfrac{\sh w }{w}\,u \ch w-w \sh u \sh w \\
&<\sh u+u \ch u-2 u \ch w\le\sh u-u \ch u<0, 
\end{align*}
so that $d|_{v=w}$ is decreasing in $u>0$, whereas $d|_{v=w,u=0+}=0$. 
Thus, indeed $d|_{v=w}<0$, and \eqref{eq:ineq} is proved in Case~2 as well. 

\bigskip

It remains to consider 

\subsubsection*{Case~3} 
Note that $\frac{\sh w}w$ is increasing in $w>0$. So, by \eqref{eq:ineq}, $d$ is decreasing in $w\in[v,\infty)$, because $g_j(u)$ and $g_j(v)$ do not depend on $w$ as long as $w\ge u\vee v$. It follows that in Case~3 w.l.o.g.\ $w=v$. Thus, $0<u\le w=v$, so that Case~3 has been quickly reduced to the already considered Case~2. 

Now inequality \eqref{eq:ineq} and thereby Theorem~\ref{th:} are completely proved. 
\end{proof}

\bibliographystyle{acm}
\bibliography{C:/Users/Iosif/Dropbox/mtu/bib_files/citations}

\end{document}